\documentclass[11pt,a4paper]{amsart}
\usepackage[T1]{fontenc}
\usepackage[utf8]{inputenc}
\usepackage{lmodern}
\usepackage{amsmath,amssymb,amsthm,mathrsfs}
\usepackage{geometry}
\usepackage{microtype}
\usepackage{enumitem,needspace}
\usepackage{cite}
\usepackage[hidelinks]{hyperref}
\hypersetup{pdftitle={Mixed sums of infinite multiple zeta-star values with restricted indices},
  pdfsubject={Revised manuscript}}
\allowdisplaybreaks[1]
\numberwithin{equation}{section}
\theoremstyle{plain}
\newtheorem{Thm}{Theorem}[section]
\newtheorem{lem}[Thm]{Lemma}
\newtheorem{prop}[Thm]{Proposition}

\newtheorem{rem}[Thm]{Remark}

\DeclareMathOperator{\conv}{conv}
\DeclareMathOperator{\diam}{diam}
\DeclareMathOperator{\maxgap}{maxgap}

\makeatletter
\renewenvironment{proof}[1][\proofname]{%
  \par
  \pushQED{\qed}%
  \normalfont
  \topsep6\p@\@plus6\p@\relax
  \trivlist
  \item[\hskip\labelsep\bfseries #1\@addpunct{:}]%
  \ignorespaces
}{%
  \popQED
  \endtrivlist
  \@endpefalse
}
\makeatother
\begin{document}
\title{Mixed sums of infinite multiple zeta-star values with restricted indices}

\author{Jiangtao Li}

\email{lijiangtao@csu.edu.cn}
\address{Jiangtao Li \\ School of Mathematics and Statistics, HNP-LAMA, Central South University, Hunan Province, China}

\author{Siyu Yang*}

\email{siyuyang@csu.edu.cn}
\address{Siyu Yang \\ School of Mathematics and Statistics, HNP-LAMA, Central South University, Hunan Province, China  }

\begin{abstract}
We study arithmetic sums of infinite multiple zeta-star values subject to
upper or lower bounds on their indices. Let $\eta(\mathcal D_q)$ be the image, under the
zeta-star correspondence, of the admissible infinite index sequences with all indices at most $q$,
and let $\overline{\eta(\mathcal T_p)}$ be the closure of the image of the sequences with all indices
at least $p$. We prove that $\eta(\mathcal D_q)+\overline{\eta(\mathcal T_p)}$ is an interval exactly when $p\leq q$;
in that case it is $[\zeta^\star(\{q\}^{\infty})+1,\infty)$.
We also show that sums of two upper-bounded sets are half-lines, even with
different bounds, whereas sums of two lower-bounded sets are never intervals.
The proofs combine a uniform product estimate with natural gap
constructions. Lateral thickness handles the diagonal case $p=q\geq3$,
while a non-strict local bridge--gap criterion treats $p=q=2$ directly.
Compact truncations and
explicit omitted intervals yield the remaining assertions.
\end{abstract}

\makeatletter
\begingroup
\renewcommand{\@makefntext}[1]{\noindent #1}
\let\thefootnote\relax
\footnotetext{%
2020 $\mathnormal{Mathematics}\;\mathnormal{Subject}\;\mathnormal{Classification}$. 11M32, 11B13, 28A80.\\
$\mathnormal{Keywords:}$ Infinite multiple zeta-star values, arithmetic sumsets, interval property, Cantor sets.\\
*Corresponding author%
}
\endgroup
\makeatother

\maketitle

\section{Introduction}\label{sec:intro}

For integers $r\geq1$, $k_1\geq2$, and $k_2,\ldots,k_r\geq1$, the multiple
zeta value and the multiple zeta-star value are, respectively,
\[
\zeta(k_1,\ldots,k_r)
 =\sum_{n_1>\cdots>n_r\geq1}\frac1{n_1^{k_1}\cdots n_r^{k_r}},
\qquad
\zeta^\star(k_1,\ldots,k_r)
 =\sum_{n_1\geq\cdots\geq n_r\geq1}\frac1{n_1^{k_1}\cdots n_r^{k_r}}.
\]
Besides their algebraic relations, these values have order-theoretic and
topological structures. Kumar~\cite{kumar} determined the iterated derived
sets and the order type of the set of multiple zeta values. Li and
Pan~\cite{lipan} studied related topological questions for certain
$q$-analogues.

Denote by $\mathbb{Z}^+$ the set of positive integers. For multiple zeta-star values, the first author \cite{lit} established a
natural order structure for multiple zeta-star
values and a bijection $\eta$  from a suitable space of infinite index sequences to $(1,+\infty)$.  It is 
\[
\eta:\mathcal T\rightarrow(1,+\infty),
\]
\[
(k_1,k_2,\cdots)\mapsto \lim_{r\rightarrow +\infty}  \zeta^\star(k_1,\cdots,k_r),
\]
where
\[
\mathcal T=\Big\{(k_1,k_2,\ldots)\in(\mathbb{Z}^+)^{\infty}\;\Big|\; k_1\geq2,(k_1, k_2,\cdots)\neq (2,\{1\}^{\infty})\Big\}.
\]
We also write $\eta(k_1,k_2,\ldots)=\zeta^\star(k_1,k_2,\ldots)$.
Thus every real number greater than $1$ has a {\bf unique representation} of
this form, and restrictions on the indices define subsets of $(1,\infty)$.

Independently, Hirose, Murahara, and Onozuka~\cite{hmo} developed
multiple zeta-star values of infinite length, obtained the corresponding
parametrization, and studied explicit derivative formulas. Further developments by the first author include rational deformations and their relation to Cantor
sets \cite{lir}, Diophantine approximation of multiple zeta-star
values \cite{lida}, and a generalization to the polylogarithm star
correspondence \cite{lipoly}.

For integers $p,q\geq2$, set
\[
\begin{aligned}
\mathcal D_q&=\{(k_j)\in\mathcal T:k_j\leq q\text{ for every }j\},
& D_q&:=\eta(\mathcal D_q),\\
\mathcal T_p&=\{(k_j)\in\mathcal T:k_j\geq p\text{ for every }j\},
& T_p&:=\overline{\eta(\mathcal T_p)}.
\end{aligned}
\]

In our previous work \cite{LY}, we studied arithmetic sums and products of these
restricted sets. In particular, we proved
\[
\eta(\mathcal D_q)+\eta(\mathcal D_q) = D_q+D_q=[2 \zeta^\star(\{q\}^{\infty}) ,\infty)\qquad(q\geq2),
\]
and showed that $\overline{\eta(\mathcal T_p)}+\overline{\eta(\mathcal T_p)} = T_p+T_p $ is not an interval for any $p\geq2$.
These results suggest three mixed problems: different upper bounds,
different lower bounds, and an upper bound paired with a lower bound.
These problems are the subject of this paper. 

The similarity with continued fractions is useful here. An irrational real
number has a continued fraction expansion
\[
 x=[a_0;a_1,a_2,\ldots]
 =a_0+\cfrac1{a_1+\cfrac1{a_2+\cfrac1{\ddots}}},
 \qquad a_0\in\mathbb Z,\quad a_j\in\mathbb Z^+\ (j\geq1),
\]
while rational numbers have finite expansions. Restrictions on the partial
quotients produce sets with a natural Cantor structure. Let $F(m)$ denote
the set for which $a_j\leq m$ for all $j\geq1$, with arbitrary integer part.
Using direct comparisons between deleted gaps and retained intervals,
Hall~\cite{hall} proved in 1947 that
\[
 F(4)+F(4)=\mathbb R,
 \qquad F(4)\cdot F(4)\supset[1,+\infty).
\]
For the opposite restriction, let $G(2)$ denote the set with $a_j\geq2, j\geq 1$.
Cusick~\cite{cusick} proved in 1971 that
\[
 \overline{G(2)}+\overline{G(2)}=\mathbb R,
 \qquad
 \overline{G(2)}\cdot\overline{G(2)}\supset[1,+\infty).
\]
Hlavka~\cite{hlavka} considered different upper bounds and established,
among other results, $F(3)+F(4)=\mathbb R$ and $F(2)+F(7)=\mathbb R$ in 1975.

A general geometric approach arose from Newhouse's work on homoclinic
bifurcations. Newhouse~\cite{new} introduced thickness and proved the Gap
Lemma in 1979. Thickness compares bridges with adjacent gaps. For compact
Cantor sets $C_1,C_2$ whose convex hulls meet, and neither of which is
contained in a bounded gap of the other, the strict inequality
\[
 t(C_1)t(C_2)>1
\]
forces $C_1\cap C_2\ne\varnothing$. The relation
\[
 x\in C_1+C_2\quad\Longleftrightarrow\quad
 C_1\cap(x-C_2)\ne\varnothing
\]
then turns intersection criteria into results about arithmetic sums.
Astels~\cite{astels} developed thickness measures for finite sums of
Cantor sets in 1999. The associated interval criterion also includes the
boundary case $t(C_1)t(C_2)=1$. His subsequent work~\cite{astelscf}
applied these ideas to continued fractions with restricted partial
quotients, extending the line of Hall and Hlavka.

Several refinements are relevant to this perspective. Moreira~\cite{moreira}
used gap presentations to distinguish left and right thicknesses in 1996,
obtaining an asymmetric Gap Lemma that retains information lost by ordinary
thickness. 

Palis~\cite{palis} conjectured that, generically, the arithmetic
difference of two regular Cantor sets either has zero Lebesgue measure or
contains an interval. 
Moreira and Yoccoz~\cite{moreirayoccoz}, with the corrected proof of the scale recurrence lemma given in~\cite{moreirazamudio}, proved this conjecture in the generic setting in 2001. In particular, when the sum of the dimensions is greater than $1$, their result gives stable intersections after suitable translations for generic pairs of regular Cantor sets. 

Cabrelli, Hare, and Molter~\cite{chm} developed a local criterion
in 2002 that compares bridge--gap pairs from the two sets directly, making
it useful when local lengths are explicitly available. 

More recently,
Takahashi~\cite{takahashi} proved in 2019 that two homogeneous Cantor sets
whose Hausdorff dimensions sum to more than $1$ admit arbitrarily small
perturbations for which their sum contains an interval. Jiang \cite{jiang} further developed thickness criteria for nonlinear images of Cantor sets. These results place the present questions within the wider study of intervals in sums and differences of Cantor sets.

The problem considered in this paper is to determine, for all integers
$p,q\geq2$, whether the mixed arithmetic sum of 
$\eta(\mathcal D_q)$ and $\overline{\eta(\mathcal T_p)}$
is an interval and, if so, determine it explicitly. Our
results give a complete classification according to the relative positions of
$p$ and $q$.

The two sets in this mixed sum have different geometric features:
$\eta(\mathcal D_q)$ is closed and unbounded, whereas
$\overline{\eta(\mathcal T_p)}$ is compact. This asymmetry makes the mixed
problem different from the self-sum problems in \cite{LY} and requires different Cantor set methods.

\begin{Thm}\label{thm:main-positive}
For integers $2\leq p\leq q$,
\[
\eta(\mathcal D_q)+\overline{\eta(\mathcal T_p)}
=\left[\zeta^\star(\{q\}^{\infty})+1,+\infty\right).
\]
\end{Thm}

\begin{Thm}\label{thm:main-negative}
For integers $p,q\geq2$ with $p\geq q+1$, the set 
\[
\eta(\mathcal D_q)+\overline{\eta(\mathcal T_p)}
\] is not an interval.
\end{Thm}

Thus $\eta(\mathcal D_q)+\overline{\eta(\mathcal T_p)}$ is an interval if and only if $p\leq q$. The remaining two
results extend the self-sum theorems of~\cite{LY} to different bounds.

\begin{Thm}\label{thm:mixed-upper}
For all integers $q_1,q_2\geq2$,
\[
 \eta(\mathcal D_{q_1})+\eta(\mathcal D_{q_2})
 =
 \left[\zeta^\star(\{q_1\}^{\infty})
 +\zeta^\star(\{q_2\}^{\infty}),+\infty\right).
\]
\end{Thm}

\begin{Thm}\label{thm:mixed-lower}
For all integers $p_1,p_2\geq2$, the set 
\[
\overline{\eta(\mathcal T_{p_1})}
+
\overline{\eta(\mathcal T_{p_2})}
\]
 is not an interval.
\end{Thm}

In particular, every sufficiently large real number is a sum of two
infinite multiple zeta-star values with prescribed, possibly different,
upper bounds on their indices. Thus the mixed sum problem for these two
classes of restricted infinite multiple zeta-star values is completely settled. The negative assertions mean that the whole sum is not an
interval; they do not assert that it contains no intervals.
Corresponding mixed products remain a natural direction for further study.

For $p\geq q+1$, we construct a natural gap of $D_q$
longer than $\diam(T_p)$. This gives an open interval disjoint from
$D_q+T_p$.

For $p=q=2$, an exact tail formula isolates one exceptional bridge of $T_2$, while a separation of the possible gap lengths of $D_2$ allows the local Cabrelli--Hare--Molter criterion to be applied directly. When $p=q\geq3$, the diagonal identity follows from two lateral thickness estimates. If $2\leq p<q$, we have the inclusion $\mathcal T_q\subseteq\mathcal T_p$. So the result follows at once from the diagonal case.

Finally, the upper--upper formula follows from compact truncations and the Newhouse--Astels interval criterion, while the lower--lower case is obtained from the initial subdivision of the corresponding lower-bounded set.

Section~\ref{sec:prelim} develops the common estimates and gap criteria.
Section~\ref{sec:non-diagonal} proves the negative result.
Section~\ref{sec:positive-range} establishes the diagonal cases and hence
the positive classification. Section~\ref{sec:mixed} proves the two
remaining mixed sum results.

\subsection*{AI disclosure}
Artificial intelligence tools  were used to assist with literature searches,
identifying proof methods, organizing the exposition, carrying out
calculations, and checking the arguments for possible gaps.

\section{Notation and preliminaries}\label{sec:prelim}

\subsection{Formulas and natural subdivisions}
We follow the notation and natural constructions in~\cite{LY,lit,lir,hmo}.
The compact closure $T_p$ is
\begin{equation}\label{eq:Tp-closure}
 T_p :=\overline{\eta(\mathcal T_p)}=\eta(\mathcal T_p)\cup
 \{\zeta^\star(k_1,\ldots,k_r):r\geq1,\ k_j\geq p\}
 \cup\{1\}.
\end{equation}
By~\cite[Theorem~1.4(iii)]{lit}, $D_q:=\eta(\mathcal D_q)$ is relatively closed in
$(1,+\infty)$. We use the endpoint notation
\begin{equation}\label{eq:As}
	A_s:=\zeta^\star(\{s\}^{\infty})
	=\prod_{n=2}^{\infty}(1-n^{-s})^{-1},\qquad s\geq2.
\end{equation}

 Since $D_q\subset[A_q,+\infty)$ and $A_q>1$, it is
also closed in $\mathbb R$. Its natural construction in
\cite[Section~4]{LY} shows that it is unbounded, perfect, and nowhere dense. The set $T_p:=\overline{\eta(\mathcal T_p)}$ is a compact Cantor set. The correspondence reverses
lexicographic order: at the first differing index, the smaller entry gives
the larger value~\cite{lit,LY}.  Their endpoints are
\begin{equation}\label{eq:endpoints}
 \min D_q=A_q,\qquad \conv(T_p)=[1,A_p],\qquad
 \diam(T_p)=A_p-1,\qquad A_2=2.
\end{equation}
Here and below, a compact Cantor set is nonempty, compact, perfect, and
nowhere dense.

For $s\geq1$ and $m\geq1$, define
\[
 F_m(s):=\prod_{u=2}^{m}(1-u^{-s})^{-1},\qquad F_1(s)=1.
\]

A nonempty prefix $P=(k_1,\ldots,k_r)$ is admissible for
$\mathcal D_q$ if $2\leq k_1\leq q$ and $1\leq k_j\leq q$ for
$j\geq2$, and for $\mathcal T_p$ if $k_j\geq p$ for every $j$.

For a finite admissible prefix $P=(k_1,\ldots,k_r)$, put
\[
 \Lambda_P(m):=\sum_{n_1\geq\cdots\geq n_r\geq m}
 \frac1{n_1^{k_1}\cdots n_r^{k_r}},\qquad
 \Lambda_{\emptyset}(m)=1.
\]
The coefficients are positive and non-increasing in $m$. For an integer
$s\geq1$, whenever the infinite index is admissible, the tail formula is
\begin{equation}\label{eq:tailformula}
 \eta(P,i,\{s\}^{\infty})
 =\sum_{m\geq1}\Lambda_P(m)m^{-i}F_m(s).
\end{equation}
We also use
\begin{equation}\label{eq:height-one}
 F_m(1)=m,\qquad
 h_r:=\zeta^\star(2,\{1\}^{r})=(r+1)\zeta(r+2)\quad(r\geq0).
\end{equation}
Here is a derivation of the identities used repeatedly below. For a fixed
upper summation limit $m$, a tail of $\ell$ copies of $s$ contributes
a sum over multisets from $1,\ldots,m$. The multiplicities of
$2,\ldots,m$ may be any nonnegative integers with total at most
$\ell$; all remaining entries are $1$. Monotone convergence as
$\ell\to+\infty$ therefore gives a product of geometric series,
namely $F_m(s)$. Summing over the entry $m$ immediately after the
prefix proves~\eqref{eq:tailformula}, again by monotone convergence.
For $s=1$, it is clear that $F_m(1)=m$. Hence
\begin{equation}\label{eq:finite-endpoint}
 \eta(P,i+1,\{1\}^{\infty})=\zeta^\star(P,i)
\end{equation}
whenever $(P,i)$ is a finite admissible index. For $s\geq2$, the
identity $m^{-s}F_m(s)=F_m(s)-F_{m-1}(s)$ for $m\geq2$, applied
with $P=\emptyset$ and $i=s$, also proves the product representation
of $A_s$ in~\eqref{eq:As}. The exceptional infinite index
$(2,\{1\}^{\infty})$ has value $+\infty$. No subtraction of two
divergent values is used in this paper.

The height-one identity in~\eqref{eq:height-one}
follows from the classical sum formula, proved independently
by Granville~\cite{gran} and Zagier~\cite{zag},
after expanding the star value into ordinary multiple zeta values.
For completeness, we give a generating-function proof below.
For $0\leq t<1$, the beta integral
gives
\[
 \frac1n\prod_{j=1}^n(1-t/j)^{-1}
 =\int_0^1 x^{n-1}(1-x)^{-t}\,dx.
\]
Consequently, expanding nonnegative series and then setting $u=1-x$,
\begin{align*}
 \sum_{r\geq0}h_rt^r
 &=\sum_{n\geq1}\frac1{n^2}\prod_{j=1}^n(1-t/j)^{-1}
 =\int_0^1\frac{-u^{-t}\log u}{1-u}\,du\\
 &=\sum_{n\geq1}\frac1{(n-t)^2}
 =\sum_{r\geq0}(r+1)\zeta(r+2)t^r.
\end{align*}
The integrals and series are finite for $0\leq t<1$, so coefficient
comparison proves the formula for $h_r$. 

For an admissible nonempty prefix $P$ of $\mathcal D_q$ and
$1\leq i\leq q-1$, the construction interval corresponding to next
indices $i,\ldots,q$ is subdivided by removing the natural gap
\begin{equation}\label{eq:Egaps}
	\mathcal G_{D,q}(P,i)
	=
	\bigl(\eta(P,i+1,\{1\}^{\infty}),
	\eta(P,i,\{q\}^{\infty})\bigr).
\end{equation}
The left and right retained intervals are
\[
\bigl[\eta(P,\{q\}^{\infty}),
\eta(P,i+1,\{1\}^{\infty})\bigr]
\quad\text{and}\quad
\bigl[\eta(P,i,\{q\}^{\infty}),
\eta(P,i,\{1\}^{\infty})\bigr],
\]
respectively. If $(P,i,\{1\}^{\infty})=(2,\{1\}^{\infty})$, the
right retained interval is
$[\eta(P,i,\{q\}^{\infty}),+\infty)$.
For $P=\emptyset$, the same formulas apply with $2\leq i\leq q-1$.
When $q=2$, the first subdivision has $P=(2)$ and $i=1$.

For an admissible nonempty prefix $P$ of $\mathcal T_p$ and $i\geq p$,
the construction interval corresponding to next indices at least $i$
is subdivided by removing the natural gap
\begin{equation}\label{eq:Tgaps}
	\mathcal G_{T,p}(P,i)
	=
	\bigl(\eta(P,i+1,\{p\}^{\infty}),
	\zeta^\star(P,i)\bigr).
\end{equation}
The left and right retained intervals are
\[
\bigl[\zeta^\star(P),
\eta(P,i+1,\{p\}^{\infty})\bigr]
\quad\text{and}\quad
\bigl[\zeta^\star(P,i),
\eta(P,i,\{p\}^{\infty})\bigr],
\]
respectively. At the initial level, $P=\emptyset$, $i\geq p$, and
$\zeta^\star(\emptyset)=1$.

Write $M,G,N$ for the lengths of the left retained interval, the gap,
and the right retained interval, respectively. An unbounded retained
interval has length $+\infty$. By~\eqref{eq:tailformula}, for $D_q$
these lengths are
\begin{equation}\label{eq:E-M}
	M_{D,q}(P,i)
	=
	\sum_{m\geq2}\Lambda_P(m)m^{-i}
	\bigl(1-m^{i-q}F_m(q)\bigr),
\end{equation}
\begin{equation}\label{eq:E-G}
	G_{D,q}(P,i)
	=
	\sum_{m\geq2}\Lambda_P(m)m^{-i}
	\bigl(F_m(q)-1\bigr),
\end{equation}
\begin{equation}\label{eq:E-N}
	N_{D,q}(P,i)
	=
	\sum_{m\geq2}\Lambda_P(m)m^{-i}
	\bigl(m-F_m(q)\bigr).
\end{equation}
For $T_p$, they are
\begin{equation}\label{eq:T-M}
	M_{T,p}(P,i)
	=
	\sum_{m\geq2}\Lambda_P(m)m^{-i}\frac{F_m(p)}{m},
\end{equation}
\begin{equation}\label{eq:T-G}
	G_{T,p}(P,i)
	=
	\sum_{m\geq2}\Lambda_P(m)m^{-i}
	\left(1-\frac{F_m(p)}{m}\right),
\end{equation}
\begin{equation}\label{eq:T-N}
	N_{T,p}(P,i)
	=
	\sum_{m\geq2}\Lambda_P(m)m^{-i}
	\bigl(F_m(p)-1\bigr).
\end{equation}
Repeating these subdivisions over all admissible finite prefixes gives
$D_q$ and $T_p$. Their bounded gaps are therefore exactly the natural
gaps above~\cite[Section~4]{LY}.

\subsection{A uniform product estimate}
\begin{lem}\label{lem:product}
For $s\geq2$, set $c_s=2^{s-1}-1$. Then
\begin{equation}\label{eq:product-bound}
 F_m(s)\leq\frac{(c_s+1)m}{c_sm+1}\qquad(m\geq1).
\end{equation}
Moreover,
\begin{equation}\label{eq:product-limit}
 F_m(2)=\frac{2m}{m+1},\qquad
 A_s-1<\frac1{c_s}\quad(s>2).
\end{equation}
\end{lem}
\begin{proof}
Convexity of $x^{s-1}$, applied to the secant through $1$ and $2$, gives
\[
 n^{s-1}\geq c_s(n-1)+1\qquad(n\geq2).
\]
Consequently,
\[
 1-\frac1{n^s}
 \geq1-\frac1{n[c_s(n-1)+1]}
 =\frac{(c_sn+1)(n-1)}{n[c_s(n-1)+1]}.
\]
Multiplication from $n=2$ to $m$ telescopes and proves
\eqref{eq:product-bound}. For $s=2$ every factor is an equality. For
$s>2$, the ratio of the two factors at $n=3$ is a fixed number
$\rho_s>1$. Keeping this factor in the telescoping product gives
$A_s\leq(c_s+1)/(\rho_sc_s)<(c_s+1)/c_s$ after passage to the limit.
This proves~\eqref{eq:product-limit}.
\end{proof}

\begin{lem}\label{lem:E-estimates}
Every natural gap of $D_q$ satisfies
\begin{equation}\label{eq:E-bridges}
 M_{D,q}\geq c_qG_{D,q},\qquad
 N_{D,q}\geq2c_qG_{D,q},\qquad
 G_{D,q}<A_q-1.
\end{equation}
For $q=2$, the first inequality is an equality.
\end{lem}
\begin{proof}
Lemma \ref{lem:product} is equivalent to
\[
 c_q( F_m(q) -1)\leq1- \frac{F_m(q)}{m} \leq1-m^{i-q} F_m(q),
\]
which gives the pointwise bound needed for the first inequality
in~\eqref{eq:E-bridges}. For the second inequality, observe that
\[
 \frac{m+2c_q}{2c_q+1}-\frac{(c_q+1)m}{c_qm+1}
 =\frac{c_q(m-1)(m-2)}{(2c_q+1)(c_qm+1)}\geq0.
\]
Thus $m- F_m(q) \geq2c_q(F_m(q)-1)$. Summation in~\eqref{eq:E-M}--\eqref{eq:E-N} proves
the first two estimates. When $q=2$, substituting $F_m(2)=2m/(m+1)$ gives
$M_{D,2}=G_{D,2}$.

For the gap bound, $F_m(q)<A_q$ for every finite $m$, and hence
\[
 G_{D,q}(P,i)<(A_q-1)\sum_{m\geq2}\Lambda_P(m)m^{-i}.
\]
If $P$ has length $r\geq1$, termwise comparison with the prefix
$(2,\{1\}^{r-1})$ gives
\[
 \sum_{m\geq2}\Lambda_P(m)m^{-i}
 \leq h_r-h_{r-1}
 =1+\sum_{n\geq2}\frac{(r+1)-rn}{n^{r+2}}<1.
\]
For the empty prefix, $i\geq2$ and the same sum is at most
$\zeta(2)-1<1$. This proves the final assertion.
\end{proof}

\subsection{Prefix bounds and a distinguished family of gaps}
\begin{lem}\label{lem:prefix}
For $r\geq0$, let $P_r=(2,\{1\}^{r})$. Then, for $m\geq2$,
\begin{equation}\label{eq:prefix-spine}
\frac1m<\Lambda_{P_r}(m)<\frac1{m-1},\qquad
\Lambda_{P_{r+1}}(m)>\Lambda_{P_r}(m).
\end{equation}
If $P$ is any nonempty admissible upper-bounded prefix not equal to a
$P_r$, then
\begin{equation}\label{eq:prefix-small}
 \Lambda_P(m)\leq\frac1{m(m-1)}\qquad(m\geq2).
\end{equation}
\end{lem}
\begin{proof}
Appending an index gives the recursion
\begin{equation}\label{eq:prefix-recursion}
 \Lambda_{(P,j)}(m)=\sum_{n\geq m}\Lambda_P(n)n^{-j}.
\end{equation}
Integral comparison for the decreasing function $x^{-2}$ gives
\[
\frac{1}{m}
<
\Lambda_{P_0}(m)
=
\sum_{n\ge m}\frac{1}{n^2}
<
\frac{1}{m-1}.
\]
Since
\[
\Lambda_{P_1}(m) >\sum_{n^\prime\geq n\geq m}\frac{1}{n^\prime(n^\prime+1) n}=\sum_{n\geq m}\frac{1}{n^2}=\Lambda_{P_0}(m),\]
the recursion proves~\eqref{eq:prefix-spine} by induction.

For the second assertion, use
\begin{equation}\label{eq:prefix-kernel}
 \sum_{n\geq m}\frac1{n^2(n-1)}
 <\sum_{n\geq m}\frac2{(n-1)n(n+1)}
 =\frac1{m(m-1)}.
\end{equation}
If the first index is at least $3$, the same bound follows from
$$\frac{1}{n^{3}}<\frac{2}{(n-1)n(n+1)}.$$ Otherwise, the first departure from a prefix
$P_r$ appends an index $j\geq2$, so~\eqref{eq:prefix-spine} and
\eqref{eq:prefix-kernel} give~\eqref{eq:prefix-small}. Once obtained,
this estimate is preserved by every further index $j\geq1$, again by
\eqref{eq:prefix-recursion} and~\eqref{eq:prefix-kernel}.
\end{proof}

The gaps corresponding to $(P_r,1)$ will be used repeatedly. Write
\begin{equation}\label{eq:spine-gaps}
 \begin{aligned}
 U_r&=\eta(2,\{1\}^{r},2,\{1\}^{\infty})=h_{r+1},\\
 V_{r,q}&=\eta(2,\{1\}^{r+1},\{q\}^{\infty}),\\
 \Delta_{r,q}&=V_{r,q}-U_r
   =\sum_{m\geq2}\frac{\Lambda_{P_r}(m)}m\bigl(F_m(q)-1\bigr).
 \end{aligned}
\end{equation}
By~\eqref{eq:prefix-spine},
\begin{equation}\label{eq:increasing-gaps}
 0<\Delta_{r,q}<\Delta_{r+1,q}.
\end{equation}

\subsection{Bridges, thickness, and two half-line criteria}
A bounded gap of a closed set $C$ is a bounded connected component of
$\conv(C)\setminus C$. For a compact Cantor set, write $\maxgap(C)$ for
the supremum of its gap lengths. For a gap $G=(a,b)$, 
the \emph{left Newhouse bridge} $B_L^C(G)$ of $G$ is the maximal closed interval in
$\operatorname{conv}(C)$ having $a$ as its right endpoint and containing no
point of any gap $G'$ with $|G'|\geq|G|$. The \emph{right Newhouse bridge} $B_R^C(G)$ is defined symmetrically.
Either interval is called a \emph{Newhouse bridge} of $G$. 

A pair
$(B,G)$, where $B$ is either Newhouse bridge associated with $G$, is
called a \emph{Newhouse bridge--gap pair}.
Infinite lengths
are allowed for an unbounded set. 
For a compact Cantor set $C$, its Newhouse thickness is
\[
 t(C)=\inf_G\min_{\sigma\in\{L,R\}}
              \frac{|B_\sigma^C(G)|}{|G|}.
\]

We will use a simple bridge comparison. Suppose that both retained
intervals at every natural subdivision have length at least $c$ times
the deleted gap. Then every Newhouse bridge has the same lower bound.
Indeed, a bridge either contains the adjacent retained interval, or is
stopped by a later gap $G'$ inside it with $|G'|\geq|G|$. In the latter
case, the retained interval of $G'$ facing $G$ lies in the bridge and
has length at least $c|G'|\geq c|G|$. In particular,
Lemma~\ref{lem:E-estimates} implies
\begin{equation}\label{eq:E-Newhouse}
 |B_\sigma^{D_q}(G)|\geq c_q|G|\qquad(\sigma=L,R).
\end{equation}

\begin{Thm}[Newhouse--Astels interval criterion; {\cite[Theorem~2.2]{astels}}]
\label{thm:newhouse}
Let $C_1,C_2$ be compact Cantor sets. If
\[
 \maxgap(C_1)\leq\diam(C_2),\qquad
 \maxgap(C_2)\leq\diam(C_1),\qquad t(C_1)t(C_2)\geq1,
\]
then $C_1+C_2$ is its full convex hull.
\end{Thm}
The non-strict thickness hypothesis is intentional: Astels's theorem
includes equality. This boundary case will be needed when both upper
bounds equal $2$.

For lateral estimates we need a different notion of bridge. A
\emph{gap presentation} $\mathcal U=(U_1,U_2,\ldots)$ enumerates the
bounded gaps. At the moment $U_n$ is removed, let $I_n$ be the connected component
of $\conv(C)\setminus\bigcup_{j<n}U_j$ containing it, and write
$I_n=L_n\cup U_n\cup R_n$. The intervals $L_n$ and $R_n$ are the left and right
\emph{presentation bridges}, respectively, and
\[
 \tau_L(C;\mathcal U)=\inf_n\frac{|L_n|}{|U_n|},\qquad
 \tau_R(C;\mathcal U)=\inf_n\frac{|R_n|}{|U_n|}.
\]
They need not be Newhouse bridges.

For the natural constructions, we use a gap presentation compatible with
the order of subdivision. More precisely, the gaps are ordered by their
depth in the natural binary subdivision tree, with any fixed order within each level, where each subdivision
separates one construction interval into two retained intervals by
removing a gap. Here the depth refers to the number of successive
subdivisions, rather than to the length of the index prefix. In
particular, successive split indices associated with the same prefix
occur at different depths along the same branch.

This gives a genuine enumeration of all natural gaps. Indeed, every
binary level contains only finitely many gaps, while every natural gap
appears after finitely many subdivisions. This remains true for $T_p$,
although its set of allowed indices is unbounded. The first subdivision
is $(P,i)=((2),1)$ for $D_2$, $(P,i)=(\emptyset,2)$ for $D_q$ with
$q\geq3$, and $(P,i)=(\emptyset,p)$ for $T_p$.

When a gap is removed, all gaps separating its construction interval
from the rest of the construction have already been removed, while
no gap inside that interval has yet been removed. Thus, at this stage,
the construction interval is exactly the component containing the gap.
Its left and right presentation bridges are therefore precisely the
left and right natural retained intervals.
We call this the \emph{natural presentation} and usually suppress it from the notation
$\tau_L,\tau_R$. In particular, the estimates in
Lemma~\ref{lem:E-estimates} give
\begin{equation}\label{eq:E-lateral}
	\tau_L(D_q)\geq c_q,\qquad
	\tau_R(D_q)\geq2c_q.
\end{equation}

The next proposition gives the two half-line forms needed below. Its
local alternative follows the argument of Proposition~4.1 in Cabrelli--Hare--Molter
\cite{chm}. Its lateral alternative follows Moreira's
presentation method in Section II, \cite{moreira}. We prove the
presentation-based version used here. The proof also makes the
unbounded setting and the non-strict local inequalities explicit.

\begin{prop}\label{prop:half-line}
Let $E$ be closed, perfect, and nowhere dense, with
$\conv(E)=[a,+\infty)$ and uniformly bounded gap lengths. Let $K$ be a
compact Cantor set whose convex hull meets $[a,+\infty)$, and assume that
$K$ is not contained in a bounded gap of $E$. Either of the following
conditions implies 
\[
E\cap K\ne\varnothing.
\]
\begin{enumerate}[label=\textup{(\alph*)},leftmargin=2em]
\item For every Newhouse bridge--gap pair $(B_E,G_E)$ of $E$ and
      $(B_K,G_K)$ of $K$,
      \[
       |B_E|\geq|G_K|\quad\text{or}\quad |B_K|\geq|G_E|.
      \]
\item There are presentations for which
      \[
       \tau_R(E)\tau_L(K)>1,\qquad
       \tau_L(E)\tau_R(K)>1.
      \]
\end{enumerate}
\end{prop}
\begin{proof}
Suppose, to the contrary, that $E\cap K=\varnothing$. Put
$\alpha=\min K$ and $\beta=\max K$. Since $E$ is closed and $K$
is compact, their distance is positive:
\[
 \delta:=\inf\{|e-k|:e\in E,\ k\in K\}>0.
\]

\smallskip
\noindent\emph{Step 1: A nonempty finite family of linked gap pairs.}
For bounded gaps $U=(u_-,u_+)$ of $E$ and $V=(v_-,v_+)$ of $K$,
we say that $(U,V)$ is linked if its endpoints have one of the
following two orders:
\[
 \begin{array}{ll}
 \textup{(I)}  & u_-<v_-<u_+<v_+,\\
 \textup{(II)} & v_-<u_-<v_+<u_+.
 \end{array}
\]
Let $\mathcal F$ be the family of all such linked pairs.

We first show that $\mathcal F\ne\varnothing$. The assumption on
the convex hulls gives $\beta\geq a$. Since $a\in E$, equality
would contradict $E\cap K=\varnothing$; hence $\beta>a$.
Consequently, $\beta$ belongs to a bounded gap $U=(u_-,u_+)$ of $E$.
Here the gap is bounded because $E$ is bounded below and unbounded
above. Since $K$ is not contained in $U$, and $\beta<u_+$ is the
largest point of $K$, there exists $k\in K$ with $k\leq u_-$.
Equality is impossible because $u_-\in E$, so
\[
 \alpha\leq k<u_-<\beta.
\]
Thus $u_-\notin K$ lies in a bounded gap $V=(v_-,v_+)$ of $K$.
As $v_+\leq\beta<u_+$, we have
\[
 v_-<u_-<v_+<u_+,
\]
which gives a member of $\mathcal F$.

Next, choose $D>0$ such that every bounded gap of $E$ has length
at most $D$. For every $(U,V)\in\mathcal F$, the two inner endpoints
belong to different sets, and hence
\[
 \begin{aligned}
  &\delta\leq u_+-v_-<\min\{|U|,|V|\}
       &&\text{in case (I)},\\
  &\delta\leq v_+-u_-<\min\{|U|,|V|\}
       &&\text{in case (II)}.
 \end{aligned}
\]
In particular, both gaps have length greater than $\delta$.
Moreover, $V\subset[\alpha,\beta]$ and $U\cap V\ne\varnothing$,
so
\[
 U\subset[\alpha-D,\beta+D].
\]
Distinct gaps of either set are disjoint. A bounded interval
contains only finitely many pairwise disjoint intervals of length
greater than $\delta$. Therefore only finitely many gaps of either
set can occur in a linked pair, and $\mathcal F$ is finite.

\smallskip
\noindent\emph{Step 2: The bridge-crossing observation.}
The following observation applies both to Newhouse bridges and to
presentation bridges. First suppose that $(U,V)$ has orientation
\textup{(I)}.

Let $J_E$ be a right bridge of $U$. If $|J_E|\geq|V|$, then
\[
 0<v_+-u_+<|V|\leq|J_E|,
\]
so $v_+$ lies strictly inside $J_E$. Since
$J_E\subset\conv(E)$ and $v_+\notin E$, there is a bounded gap
$U'=(u'_-,u'_+)$ of $E$ containing $v_+$. The point $u_+\in E$
prevents $U'$ from extending to the left of $u_+$, and hence
\[
 v_-<u_+\leq u'_-<v_+<u'_+.
\]
Thus $(U',V)$ is again linked.

Similarly, let $J_K$ be a left bridge of $V$. If $|J_K|\geq|U|$,
then
\[
 0<v_--u_-<|U|\leq|J_K|,
\]
so $u_-$ lies strictly inside $J_K\subset[\alpha,\beta]$.
Since $u_-\notin K$, it belongs to a bounded gap
$V'=(v'_-,v'_+)$ of $K$. The point $v_-\in K$ gives
\[
 v'_-<u_-<v'_+\leq v_-<u_+,
\]
and $(U,V')$ is linked.

Interchanging the roles of $E$ and $K$ gives the corresponding
statements in orientation \textup{(II)}: a left bridge of $U$ of
length at least $|V|$ replaces $U$ by the gap containing $v_-$,
whereas a right bridge of $V$ of length at least $|U|$ replaces $V$
by the gap containing $u_+$. In both cases the new pair is linked.

In every replacement, the new gap contains a point in the interior
of the chosen bridge of the old gap. If this is a Newhouse bridge,
the new gap is strictly shorter than the old one, by the definition
of that bridge. If it is a presentation bridge of a gap numbered
$n$, then no gap numbered at most $n$ meets its interior. Therefore
the new gap has presentation index strictly greater than $n$.
Notice also that the endpoint being crossed lies strictly inside
the bridge even when the bridge length equals the other gap length.
Thus non-strict bridge-length inequalities suffice for this
observation.

\smallskip
\noindent\emph{Step 3: Proof under condition \textup{(a)}.}
Choose $(U,V)\in\mathcal F$ for which $|U|+|V|$ is minimal.
In orientation \textup{(I)}, condition \textup{(a)}, applied to the
right Newhouse bridge of $U$ and the left Newhouse bridge of $V$,
gives
\[
 |B_R^E(U)|\geq|V|
 \quad\text{or}\quad
 |B_L^K(V)|\geq|U|.
\]
In orientation \textup{(II)}, apply the same condition to the
left Newhouse bridge of $U$ and the right Newhouse bridge of $V$:
\[
 |B_L^E(U)|\geq|V|
 \quad\text{or}\quad
 |B_R^K(V)|\geq|U|.
\]
In either orientation, the observation in Step~2 produces another
member of $\mathcal F$ by keeping one gap and replacing the other
by a strictly shorter gap. This contradicts the minimality of
$|U|+|V|$, proving the conclusion under \textup{(a)}.

\smallskip
\noindent\emph{Step 4: Proof under condition \textup{(b)}.}
Fix presentations satisfying \textup{(b)}. For a gap $U$ of $E$
and a gap $V$ of $K$, denote their presentation indices by
$\iota_E(U)$ and $\iota_K(V)$, respectively. Choose
$(U,V)\in\mathcal F$ maximizing
\[
 \iota_E(U)+\iota_K(V).
\]
Write $L_U,R_U$ for the left and right presentation bridges of $U$,
and $L_V,R_V$ for those of $V$.

In orientation \textup{(I)}, the definition of lateral thickness
gives
\[
 \frac{|R_U|}{|U|}\frac{|L_V|}{|V|}
 \geq\tau_R(E)\tau_L(K)>1.
\]
It follows that
\[
 |R_U|>|V|\quad\text{or}\quad |L_V|>|U|.
\]
In orientation \textup{(II)}, the other thickness inequality gives
\[
 \frac{|L_U|}{|U|}\frac{|R_V|}{|V|}
 \geq\tau_L(E)\tau_R(K)>1,
\]
and hence
\[
 |L_U|>|V|\quad\text{or}\quad |R_V|>|U|.
\]
Thus, in either orientation, Step~2 produces another linked pair
by replacing exactly one of its gaps by a gap with strictly larger
presentation index. The other gap is unchanged, contradicting the
maximality of $\iota_E(U)+\iota_K(V)$.

The assumption $E\cap K=\varnothing$ is therefore impossible under
either \textup{(a)} or \textup{(b)}. Hence $E\cap K\ne\varnothing$.
\end{proof}
\section{The case \texorpdfstring{$p\geq q+1$}{p >= q+1}}\label{sec:non-diagonal}

\begin{proof}[Proof of Theorem~\ref{thm:main-negative}]
We use the gaps $(U_r,V_{r,q})$ in~\eqref{eq:spine-gaps}. Since
$F_m(q)-1\geq F_2(q)-1=(2^q-1)^{-1}$ for $m\geq2$,
\begin{equation}\label{eq:gap-lower-bound}
 \Delta_{r,q}\geq\frac1{2^q-1}\sum_{m\geq2}\frac{\Lambda_{P_r}(m)}m
 =\frac{h_{r+1}-h_r}{2^q-1}.
\end{equation}
The height-one formula gives
\[
 h_{r+1}-h_r=(r+2)\zeta(r+3)-(r+1)\zeta(r+2)\longrightarrow1.
\]
On the other hand, Lemma \ref{lem:product} at $s=q+1$
gives
\[
 A_{q+1}-1<\frac1{2^q-1}.
\]
Choose $r$ sufficiently large that
$h_{r+1}-h_r>(2^q-1)(A_{q+1}-1)$. Since $A_s$ decreases with $s$, for every
$p\geq q+1$ we obtain
\[
 \Delta_{r,q}>A_{q+1}-1\geq A_p-1=\diam(T_p).
\]
Consequently, the interval
\[
\bigl(U_r+A_p,\,V_{r,q}+1\bigr)
\]
is nonempty, and
\begin{equation}\label{eq:omitted-upper-lower}
	\bigl(U_r+A_p,\,V_{r,q}+1\bigr)\cap(D_q+T_p)=\varnothing.
\end{equation}
Indeed, a point of $D_q$ is either at most $U_r$ or at least $V_{r,q}$,
and $T_p\subset[1,A_p]$. Moreover, $U_r,V_{r,q}\in D_q$ and
$1,A_p\in T_p$, so $U_r+A_p,\ V_{r,q}+1\in D_q+T_p$.
Hence $D_q+T_p$ is not an interval.
\end{proof}

\begin{rem}
Lemma~\ref{lem:E-estimates} gives $G_{D,q}<A_q-1$ for every gap.
Thus the diameter obstruction above cannot prove the diagonal case
$p=q$. A bridge argument is needed there.
\end{rem}

\section{The cases \texorpdfstring{$2\leq p\leq q$}{2 <= p <= q}}\label{sec:positive-range}

\subsection{The diagonal case \texorpdfstring{$p=q=2$}{p = q = 2}}\label{subsec:q2}
The identity $F_m(2)=2m/(m+1)$ makes the local geometry explicit.
By Lemma~\ref{lem:E-estimates}, every gap of $D_2$ satisfies
\begin{equation}\label{eq:E2-local}
 M_{D,2}=G_{D,2},\qquad N_{D,2}\geq2G_{D,2},\qquad
 |B_\sigma^{D_2}(G)|\geq|G|.
\end{equation}
For $T_2$ there is exactly one exceptional Newhouse bridge.

\begin{lem}\label{lem:T2-local}
	The initial gap is
	\[
	\mathcal G_0=(2\zeta(2)-2,\zeta(2)).
	\]
	Its left Newhouse bridge is
	\[
	\mathcal B_0=[1,2\zeta(2)-2],
	\]
	while its right Newhouse bridge is $[\zeta(2),2]$. Moreover,
	\begin{equation}\label{eq:exceptional-lengths}
		|\mathcal G_0|=2-\zeta(2),\qquad
		|\mathcal B_0|=2\zeta(2)-3,\qquad
		\frac14<|\mathcal B_0|<|\mathcal G_0|.
	\end{equation}
	The gap $\mathcal G_0$ is uniquely largest in $T_2$. Every Newhouse
	bridge--gap pair other than $(\mathcal B_0,\mathcal G_0)$ satisfies
	$|B|\geq|G|$.
\end{lem}

\begin{proof}
For $i\geq2$, formulas~\eqref{eq:T-M}--\eqref{eq:T-N} become
\begin{align}
 G_{T,2}(P,i)=N_{T,2}(P,i)
 &=\sum_{m\geq2}\Lambda_P(m)\frac{m-1}{m^i(m+1)},
 \label{eq:T2-G}\\
 M_{T,2}(P,i)-G_{T,2}(P,i)
 &=\sum_{m\geq2}\Lambda_P(m)\frac{3-m}{m^i(m+1)}.
 \label{eq:T2-difference}
\end{align}
At $(P,i)=(\emptyset,2)$, partial fractions give the two length
formulas in~\eqref{eq:exceptional-lengths}. The inequalities follow from the
elementary bounds $13/8<\zeta(2)<5/3$, obtainable by finite sums and
integral comparison.

We show that every other left retained interval is longer than its gap.
Set
\[
 S:=\sum_{m=4}^{\infty}\frac{m-3}{m^3(m+1)}
 <\frac1{320}+\frac1{125}+\frac1{50}<\frac1{24}.
\]
Here the first term was separated, and the tail was bounded by
\[\sum_{m\geq5}m^{-3}\leq5^{-3}+\int_5^\infty x^{-3}\,dx.\]
If $i\geq3$, monotonicity of $\Lambda_P(m)$ in
\eqref{eq:T2-difference} gives
\[
\begin{aligned}
	M_{T,2}(P,i)-G_{T,2}(P,i)
	&=\frac{\Lambda_P(2)}{3\cdot 2^i}
	-\sum_{m\geq4}\Lambda_P(m)
	\frac{m-3}{m^i(m+1)}\\
	&\geq \frac{\Lambda_P(2)}{3\cdot 2^i}
	-\Lambda_P(2)4^{-(i-3)}
	\sum_{m\geq4}\frac{m-3}{m^3(m+1)}\\
	&=\Lambda_P(2)\left(
	\frac{1}{3\cdot 2^i}-4^{-(i-3)}S
	\right)>0.
\end{aligned}
\]

It remains to treat $i=2$ and $P\ne\emptyset$. We use the tail bound
\begin{equation}\label{eq:lambda-tail}
 \frac{\Lambda_P(m)}{\Lambda_P(2)}\leq\frac2m\qquad(m\geq2).
\end{equation}
To prove it, write
\[
 \Lambda_P(m)=\sum_{n\geq m}a_n,\qquad
 a_n=n^{-k_r}\Lambda_{(k_1,\ldots,k_{r-1})}(n).
\]
Since $k_r\geq2$ and
$\Lambda_{(k_1,\ldots,k_{r-1})}(n)$ is positive and
non-increasing in $n$, the sequence
\[
 \frac{a_n}{n^{-2}}
 =n^{2-k_r}\Lambda_{(k_1,\ldots,k_{r-1})}(n)
\]
is non-increasing. Here, when $r=1$, we use the convention
$\Lambda_{\emptyset}(n)=1$. In particular,
\[
 0<a_n\leq 4a_2\,n^{-2}\qquad(n\geq2),
\]
so all the series below converge.

If $m=2$, both sides of the desired inequality are equal to $1$.
Now let $m\geq3$. For every pair of integers $j,n$ with
$2\leq j<m\leq n$, monotonicity gives
\[
 \frac{a_n}{n^{-2}}\leq\frac{a_j}{j^{-2}},
 \qquad\Leftrightarrow\qquad
 a_nj^{-2}\leq a_jn^{-2}.
\]
Summing over $2\leq j<m$ and $n\geq m$, we obtain
\[
 \left(\sum_{n\geq m}a_n\right)
 \left(\sum_{j=2}^{m-1}j^{-2}\right)
 \leq
 \left(\sum_{n\geq m}n^{-2}\right)
 \left(\sum_{j=2}^{m-1}a_j\right).
\]
Adding
\[
 \left(\sum_{n\geq m}a_n\right)
 \left(\sum_{n\geq m}n^{-2}\right)
\]
to both sides yields
\[
 \left(\sum_{n\geq m}a_n\right)
 \left(\sum_{n\geq2}n^{-2}\right)
 \leq
 \left(\sum_{n\geq m}n^{-2}\right)
 \left(\sum_{n\geq2}a_n\right).
\]
It follows that
\[
 \frac{\sum_{n\geq m}a_n}{\sum_{n\geq2}a_n}
 \leq
 \frac{\sum_{n\geq m}n^{-2}}{\sum_{n\geq2}n^{-2}}.
\]
The denominator exceeds $3/5$, while telescoping gives
\[
 \sum_{n\geq m}n^{-2}
 <\sum_{n\geq m}\left(\frac1{n-1/2}-\frac1{n+1/2}\right)
 =\frac1{m-1/2}.
\]
For $m\geq3$ the last ratio $\frac{\sum_{n\geq m}n^{-2}}{\sum_{n\geq2}n^{-2}}$ is at most $2/m$. The case $m=2$ is
immediate. This proves~\eqref{eq:lambda-tail}, and hence
\[
\begin{aligned}
	M_{T,2}(P,2)-G_{T,2}(P,2)
	&=\frac{\Lambda_P(2)}{12}
	-\sum_{m\geq4}\Lambda_P(m)
	\frac{m-3}{m^2(m+1)}\\
	&\geq \frac{\Lambda_P(2)}{12}
	-2\Lambda_P(2)\sum_{m\geq4}
	\frac{m-3}{m^3(m+1)}\\
	&=\Lambda_P(2)\left(\frac1{12}-2S\right)>0.
\end{aligned}
\]

Finally, if $P$ has positive length $r$, then for $m\geq2$,
\[
 \Lambda_P(m)\leq\prod_{j=1}^{r}\sum_{n\geq m}n^{-k_j}
 \leq(\zeta(2)-1)^r<1.
\]
Thus~\eqref{eq:T2-G} is strictly smaller than its value at
$(\emptyset,2)$. The same is true for an empty prefix and $i>2$.
The initial gap is therefore uniquely largest, and its Newhouse
bridges extend to $1$ and $2$. For every other gap, both retained
intervals are at least as long as the gap. The bridge comparison
preceding~\eqref{eq:E-Newhouse} applies: any gap stopping a bridge
inside such a retained interval is a descendant and is also
noninitial. This proves the last assertion.
\end{proof}

\begin{lem}\label{lem:E2-spectrum}
Put $\delta=2\zeta(3)-2$. Every gap $G$ of $D_2$ satisfies
\begin{equation}\label{eq:E2-spectrum}
 |G|<\frac14\quad\text{or}\quad
 |G|\geq\delta>|\mathcal G_0|.
\end{equation}
\end{lem}
\begin{proof}
For a prefix $P_r=(2,\{1\}^r)$, the gap has length
$\Delta_{r,2}\geq\Delta_{0,2}$ by~\eqref{eq:increasing-gaps}.
Writing $H_n=\sum_{m=1}^n m^{-1}$ and using~\eqref{eq:height-one},
\[
\begin{aligned}
	\Delta_{0,2}
	&=\sum_{n\geq1}\frac1{n^2}
	\left(H_n+\frac2{n+1}-2\right)\\
	&=2\zeta(3)+2[\zeta(2)-1]-2\zeta(2)=\delta.
\end{aligned}
\]
Thus every gap corresponding to a prefix $(2,\{1\}^r)$ satisfies
$|G|=\Delta_{r,2}\geq\delta$.
The elementary bounds $\zeta(2)>13/8$ and $\zeta(3)>19/16$ give
$|\mathcal G_0|=2-\zeta(2)<3/8<\delta$.

Every other prefix contains a second occurrence of $2$, so
\eqref{eq:prefix-small} gives
\[
 |G|=\sum_{m\geq2}\Lambda_P(m)\frac{m-1}{m(m+1)}
 \leq\sum_{m\geq2}\frac1{m^2(m+1)}
 =\zeta(2)-\frac32<\frac14.
\]
These alternatives exhaust the gaps of $D_2$.
\end{proof}

\begin{prop}\label{prop:q2}
One has $D_2+T_2=[3,+\infty)$.
\end{prop}
\begin{proof}
We verify the local condition in Proposition~\ref{prop:half-line}(a)
for $D_2$ and $T_2$, and therefore also for every reflected translate
of $T_2$.
For a nonexceptional bridge--gap pair $(B_T,G_T)$ of $T_2$, both
sets satisfy $|B_D|\geq|G_D|$ and $|B_T|\geq|G_T|$.
If $|G_D|\geq|G_T|$, then $|B_D|\geq|G_T|$; otherwise,
$|B_T|\geq|G_T|>|G_D|$. Thus the required alternative holds,
including when the gap lengths coincide.

For the exceptional pair $(\mathcal B_0,\mathcal G_0)$, a gap of
$D_2$ shorter than $1/4$ satisfies $|G_D|<1/4<|\mathcal B_0|$. Every other
gap satisfies
\[
|B_D|\geq|G_D|\geq\delta>|\mathcal G_0|
\]
by Lemma~\ref{lem:E2-spectrum}. This verifies the local condition
for every pair.

Fix $x>3$ and set $K_x=x-T_2$. Its convex hull is $[x-2,x-1]$,
which meets $[2,+\infty)$. Every gap of $D_2$ has length less than
$1=\diam(K_x)$, so $K_x$ cannot be contained in one of those gaps.
The half-line criterion gives $D_2\cap K_x\ne\varnothing$, hence
$x\in D_2+T_2$. Finally, $3=2+1$ belongs to the sum, and no smaller
point does.
\end{proof}

\subsection{The diagonal case \texorpdfstring{$p=q\geq3$}{p = q >= 3}}\label{subsec:qge3}
The natural presentation already gives
$\tau_L(D_q)\geq c_q$ and $\tau_R(D_q)\geq2c_q$.
The complementary bounds for $T_q$ are as follows.

\begin{lem}\label{lem:Tq-lateral}
For $q\geq3$, the natural presentation of $T_q$ satisfies
\begin{equation}\label{eq:T-lateral}
 \tau_L(T_q)\geq\frac12,\qquad
 \tau_R(T_q)\geq\beta_q,
 \qquad
 \beta_q:=\frac{2^q+3^q-1}{(2^q-1)(3^q-1)}.
\end{equation}
Moreover,
\begin{equation}\label{eq:key-product}
 2c_q\beta_q>1.
\end{equation}
\end{lem}
\begin{proof}
Let $P$ be an admissible prefix and $i\geq q$. For the right bridge,
$\beta_q=F_3(q)-1$. Thus, for $m\geq3$,
\[
 F_m(q)-1\geq\beta_q
 \geq\beta_q\left(1-\frac{F_m(q)}m\right).
\]
Recall that $c_q=2^{q-1}-1$. For $m=2$ the ratio of these two coefficients is
\[
 \frac{F_2(q)-1}{1-F_2(q)/2}=\frac1{c_q}>\beta_q.
\]
The last inequality follows on putting $x=2^q$, $y=3^q$ and noting
\[
 2(x-1)(y-1)-(x-2)(x+y-1)=x(y+1-x)>0.
\]
Summation  in~\eqref{eq:T-G} and~\eqref{eq:T-N} gives $N_{T,q}\geq\beta_qG_{T,q}$.

For the left bridge, as $F_m(q)\geq 1$, monotonicity of $\Lambda_P$ gives a direct estimate:
\[
 \begin{aligned}
 2M_{T,q}-G_{T,q}
 &=\sum_{m\geq2}\Lambda_P(m)m^{-i}
                   \left(\frac{3F_m(q)}m-1\right)\\
 &\geq\sum_{m\geq2}\Lambda_P(m)\frac{3-m}{m^{i+1}}\\
 &\geq\Lambda_P(2)
   \left(2^{-(i+1)}-\sum_{m\geq4}\frac{m-3}{m^{i+1}}\right)>0.
 \end{aligned}
\]
Indeed, for $i\geq3$,
\[
 \sum_{m\geq4}\frac{m-3}{m^{i+1}}
 <4^{-i}\left(1+\frac4{i-1}\right)<2^{-(i+1)}.
\]
The second strict inequality follows from
$1+4/(i-1)\leq3<2^{i-1}$ for $i\geq3$. This also covers the
empty prefix, for which $\Lambda_P\equiv1$.
The bridge estimates prove~\eqref{eq:T-lateral}.

Finally, $2c_q\beta_q>1$ is equivalent to
$(2^q-1)^2>3^q$. For $q\geq3$,
\[
 (2^q-1)^2\geq9\cdot4^{q-2}>9\cdot3^{q-2}=3^q,
\]
which proves~\eqref{eq:key-product}.
\end{proof}

\begin{prop}\label{prop:qge3}
For every integer $q\geq3$, one has $D_q+T_q=[A_q+1,+\infty)$.
\end{prop}
\begin{proof}
Fix $x>A_q+1$ and set $K_x=x-T_q$.
We use the presentation induced by the natural presentation of $T_q$.
Reflection reverses the left and right presentation bridges, so
\[
 \tau_L(K_x)\geq\beta_q,\qquad \tau_R(K_x)\geq\frac12.
\]
Together with~\eqref{eq:E-lateral}, these bounds give
\[
 \tau_R(D_q)\tau_L(K_x)\geq2c_q\beta_q>1,\qquad
 \tau_L(D_q)\tau_R(K_x)\geq\frac{c_q}{2}\geq\frac32>1.
\]
The convex hull $[x-A_q,x-1]$ meets $[A_q,\infty)$.
Every gap of $D_q$ is shorter than $A_q-1=\diam(K_x)$, so $K_x$
cannot lie in such a gap. Proposition~\ref{prop:half-line}(b) gives
$D_q\cap K_x\ne\varnothing$, or $x\in D_q+T_q$.
The minimum $A_q+1$ is attained, completing the proof.
\end{proof}

\begin{proof}[Proof of Theorem~\ref{thm:main-positive}]
	Propositions~\ref{prop:q2} and~\ref{prop:qge3} establish the diagonal
	case for every $q\geq2$. For $2\leq p\leq q$, we have
	$T_q\subseteq T_p$, and hence
	\begin{equation}\label{eq:diagonal-reduction}
		[A_q+1,+\infty)=D_q+T_q
		\subseteq D_q+T_p\subseteq[A_q+1,+\infty),
	\end{equation}
	where the last inclusion follows from $\min D_q=A_q$ and
	$\min T_p=1$. This proves the theorem.
\end{proof}

\section{Further mixed arithmetic sums}\label{sec:mixed}

\subsection{Compact truncations of the upper-bounded sets}
The following lemma supplies the compact sets used in
Theorem~\ref{thm:mixed-upper}. Its proof also records why truncating
$D_q$ at the chosen endpoints does not shorten its Newhouse bridges.

\begin{lem}\label{lem:compact-truncation}
	For $q\geq2$ and $R\geq1$, put
	\[
	b_R=h_R=\zeta^\star(2,\{1\}^{R}),\qquad
	K_R^{(q)}=D_q\cap[A_q,b_R].
	\]
	Then $K_R^{(q)}$ is a compact Cantor set with convex hull $[A_q,b_R]$,
	and
	\begin{equation}\label{eq:truncation-estimates}
		K_R^{(q)}\uparrow D_q,\qquad
		\maxgap(K_R^{(q)})<A_q-1,\qquad
		t(K_R^{(q)})\geq c_q=2^{q-1}-1.
	\end{equation}
\end{lem}
\begin{proof}
	The point $b_R=U_{R-1}$ is the left endpoint of a natural gap of $D_q$.
	Since $D_q$ is perfect, $b_R$ is accumulated from the left. Hence the
	closed, nowhere dense set $K_R^{(q)}$ remains perfect. Both endpoints
	belong to it. Also $b_R$ increases to infinity by~\eqref{eq:height-one},
	so these compact sets exhaust $D_q$. Their gaps are exactly the natural
	gaps of $D_q$ contained in $[A_q,b_R]$. Each has length less than
	$A_q-1$ by Lemma~\ref{lem:E-estimates}. The largest gap length of a
	compact Cantor set is attained (only finitely many gaps can exceed any
	fixed positive length), proving the strict maximal-gap bound.
	
	Let $H_R=(b_R,V_{R-1,q})$, so $|H_R|=\Delta_{R-1,q}$. We claim that
	every gap $G$ of $K_R^{(q)}$ is shorter than $H_R$. If its prefix is empty, then
	\[
	|G|\leq\sum_{m\geq2}\frac{F_m(q)-1}{m^2}<\Delta_{0,q},
	\]
	since $\Lambda_{P_0}(m)>1/m$. For a nonempty prefix other than a $P_r$,
	Lemma~\ref{lem:prefix} gives
	\[
	|G|\leq\sum_{m\geq2}\frac{F_m(q)-1}{m^2(m-1)}<\Delta_{0,q}.
	\]
	If the prefix is $P_r$, the order of the correspondence forces
	$r\leq R-1$: a prefix $P_r$ with $r\geq R$ lies entirely to the right
	of $b_R=\eta(2,\{1\}^{R-1},2,\{1\}^{\infty})$. For $r<R-1$,
	\[
	|G|\leq\Delta_{r,q}<\Delta_{R-1,q}.
	\]
	When $r=R-1$, the split $i=1$ is $H_R$ itself and is excluded; every
	split $i\geq2$ is strictly shorter than $H_R$. This proves the claim.
	
	The left Newhouse bridge of a gap of $K_R^{(q)}$ is unchanged by the
	truncation. Its right bridge cannot cross $b_R$, because $H_R$ is
	longer than that gap. Thus the right bridge is also unchanged.
	The lower bound~\eqref{eq:E-Newhouse} therefore holds in $K_R^{(q)}$,
	and taking the infimum proves the thickness estimate.
\end{proof}

\subsection{Mixed upper-bounded sets}
\begin{proof}[Proof of Theorem~\ref{thm:mixed-upper}]
Use the compact sets $K_R^{(q_1)}$ and $K_R^{(q_2)}$ from
Lemma~\ref{lem:compact-truncation}. Their thicknesses satisfy
\[
 t(K_R^{(q_1)})t(K_R^{(q_2)})\geq c_{q_1}c_{q_2}\geq1.
\]
For $s\geq2$ and $R\geq2$, their gap and diameter bounds are
\[
 \maxgap(K_R^{(s)})<A_s-1\leq1,
 \qquad \diam(K_R^{(s)})=b_R-A_s>1,
\]
since $A_s\leq2$ and $b_R\geq3\zeta(4)>3$.
Theorem~\ref{thm:newhouse} therefore gives
\[
 K_R^{(q_1)}+K_R^{(q_2)}=[A_{q_1}+A_{q_2},2b_R].
\]
As the truncations increase to the full sets and $b_R\to+\infty$,
taking their union proves the claimed half-line formula.
\end{proof}

\subsection{Mixed lower-bounded sets}
For $p\geq2$, put
\[
 B_p:=\eta(p+1,\{p\}^{\infty}).
\]
The initial subdivision gives
\begin{equation}\label{eq:Tp-initial-split}
 T_p\subset[1,B_p]\cup[\zeta(p),A_p],
\end{equation}
with the natural gap $(B_p,\zeta(p))$. We need the following quantitative
separation of its left retained interval from its gap.

\begin{lem}\label{lem:Bp-gap}
For every integer $p\geq2$, one has $2B_p<1+\zeta(p)$.
\end{lem}
\begin{proof}
For $p=2$, the identity $F_m(2)=2m/(m+1)$ gives
\[B_2=2\sum_{m\geq1}\frac{1}{m^2(m+1)}=2\zeta(2)-2.\]
The assertion follows from $\zeta(2)=\pi^2/6<5/3$.

Let $p\geq3$ and $c_p=2^{p-1}-1$. By
\eqref{eq:tailformula} and \eqref{eq:product-bound},
\begin{equation}\label{eq:Bp-difference}
 \begin{aligned}
 \frac{1+\zeta(p)}2-B_p
 &=\sum_{m\geq2}\frac1{m^p}
   \left(\frac12-\frac{F_m(p)}m\right)
 \geq\sum_{m\geq2}d_{m,p},\\
 d_{m,p}&:=\frac{c_p(m-2)-1}{2m^p(c_pm+1)}.
 \end{aligned}
\end{equation}
All these series converge absolutely. The only negative term is
$d_{2,p}=-1/[2^{p+1}(2^p-1)]$. All terms with $m\geq3$ are positive.

For $p=3$, direct calculation gives
\[
 \begin{aligned}
 \sum_{m=3}^{6}d_{m,3}
 &=\frac1{270}+\frac5{1664}+\frac1{500}+\frac{11}{8208}\\
 &>\frac{109}{12000}>\frac1{112}=|d_{2,3}|.
 \end{aligned}
\]
For $p\geq4$, only the terms $m=3,4$ are needed. Indeed, for either
of these values of $m$,
\[
 \frac{d_{m,p}}{|d_{2,p}|}
 =\left[\left(\frac4m\right)^p-\left(\frac2m\right)^p\right]
   \frac{c_p(m-2)-1}{c_pm+1}
\]
is increasing in $p$: both factors are positive and increasing,
since $c_p$ increases and
\[
 \frac{d}{dc}\left(\frac{c(m-2)-1}{cm+1}\right)
 =\frac{2(m-1)}{(cm+1)^2}>0.
\]
Consequently, comparison with $p=4$ yields
\[
 \begin{aligned}
 \frac{d_{3,p}+d_{4,p}}{|d_{2,p}|}
 &\geq480\left(\frac1{594}+\frac{13}{14848}\right)\\
 &>480\left(\frac1{600}+\frac1{1200}\right)=\frac65>1.
 \end{aligned}
\]
Thus $\sum_{m\geq2}d_{m,p}>0$ for every $p\geq3$, and
\eqref{eq:Bp-difference} proves $2B_p<1+\zeta(p)$.
\end{proof}

\begin{proof}[Proof of Theorem~\ref{thm:mixed-lower}]
By symmetry, assume $p_1\leq p_2$. Then $T_{p_2}\subseteq T_{p_1}$.
From~\eqref{eq:Tp-initial-split}, a sum of two points of $T_{p_1}$
is either at most $2B_{p_1}$ or at least $1+\zeta(p_1)$. Therefore
\[
 (2B_{p_1},\,1+\zeta(p_1))\cap(T_{p_1}+T_{p_2})=\varnothing.
\]
The interval is nonempty by Lemma~\ref{lem:Bp-gap}. On the other hand,
$2=1+1$ and $1+\zeta(p_1)$ belong to $T_{p_1}+T_{p_2}$, using the
finite value $\zeta(p_1)\in T_{p_1}$ from~\eqref{eq:Tp-closure}.
Since $B_{p_1}>1$, these two points lie on opposite sides of the
omitted open interval. The mixed sum is therefore not an interval.
\end{proof}

\begin{rem}
The last two theorems can also be approached through the Hall-type
subdivision arguments of~\cite{LY}. Here the common gap and truncation
estimates make their proofs short consequences of the preceding analysis.
The same combination of local and lateral comparisons may be useful for
other arithmetic problems involving Cantor-type sets.
\end{rem}

\section*{Acknowledgements}
This research project is funded by the National Natural Science Foundation of China (Grant No. 12571009) and the
Natural Science Foundation of Hunan Province, China (Grant No. 2026JJ40003).

\end{document}